\documentclass{amsproc}

\usepackage{graphicx}
\usepackage{subfig}

\newtheorem{theorem}{Theorem}[section]

\theoremstyle{definition}

\theoremstyle{remark}
\newtheorem{remark}[theorem]{Remark}

\numberwithin{equation}{section}

\theoremstyle{remark}

\theoremstyle{plain}
\newtheorem{thm}{\protect\theoremname}
\theoremstyle{definition}
\newtheorem{defn}{\protect\definitionname}

\makeatother

\providecommand{\definitionname}{Definition}
\providecommand{\remarkname}{Remark}
\providecommand{\theoremname}{Theorem}

\begin{document}

\title{Initial State Reconstruction on Graphs}

\author[V. A. Khoa]{Vo Anh Khoa}
\address{Department of Mathematics, Florida A\&M University, Tallahassee, FL 32307, USA}
\email{anhkhoa.vo@famu.edu}
\thanks{The first author was supported by the Faculty Research Awards Program - FAMU \#007633.}

\author[M. T. N. Truong]{Mai Thanh Nhat Truong}
\address{Department of Multimedia Engineering, Dongguk University, South Korea}
\email{mtntruong@toantin.org}

\author[I. Hogan]{Imhotep Hogan}
\address{Department of Mathematics, Florida A\&M University, Tallahassee, FL 32307, USA}
\email{imhotep1.hogan@famu.edu}

\author[R. Williams]{Roselyn Williams}
\address{Department of Mathematics, Florida A\&M University, Tallahassee, FL 32307, USA}
\email{roselyn.williams@famu.edu}

\date{\today}



\begin{abstract}
The presence of noise is an intrinsic problem in acquisition processes for digital images. One way to enhance images is to combine the forward and backward diffusion equations. However, the latter problem is well known to be exponentially unstable with respect to any small perturbations on the final data. In this scenario, the final data can be regarded as a blurred image obtained from the forward process, and that image can be pixelated as a network. Therefore, we study in this work a regularization framework for the backward diffusion equation on graphs. Our aim is to construct a spectral graph-based solution based upon a cut-off projection. Stability and convergence results are provided together with some numerical experiments.
\end{abstract}

\maketitle

\section{Statement of the initial state reconstruction on graphs}

Regularization of the restoration problem is a denoising technique used in the hope that it can
retain the image crucial signal features in the presence of noise (motion artifacts, blurring, distortion)
during procurement course. It is worth mentioning that even modern cameras which
are able to acquire high resolution images still, nowadays, produce noisy outputs.
Noise in those images usually constitutes high frequencies; therefore, one way to denoise the
images is to smooth “away” the noise. It is well known that one of the standard filtering
processes is Gaussian smoothing. In particular, let $g_0$ be a gray scale image as a real-valued
mapping; see Figure \ref{fig:1} for an overview of the surface representation of a gray scale
image. Let $g_0^{\varepsilon}$ be its noisy image, where $\varepsilon\in (0,1)$ represents the noise level. Suppose that these $g_0, g_0^{\varepsilon}$ belong to $L^2(\mathbb{R}^{2})$. Then, the Gaussian smoothing constructs a smoothed
version of $g_0^{\varepsilon}$ by the convolution $G_{\sigma}*g_0^{\varepsilon}\in C^{\infty}(\mathbb{R}^2)$,
\begin{align*}
\left(G_{\sigma}*g_{0}^{\varepsilon}\right)\left(x\right)=\int_{\mathbb{R}^{2}}\frac{1}{2\pi\sigma^{2}}e^{-\left|x-y\right|^{2}/\left(2\sigma^{2}\right)}g_{0}^{\varepsilon}\left(y\right)dy,\quad\sigma>0.
\end{align*}

\begin{figure}
	\begin{centering}
		\hfill\includegraphics[height=4cm]{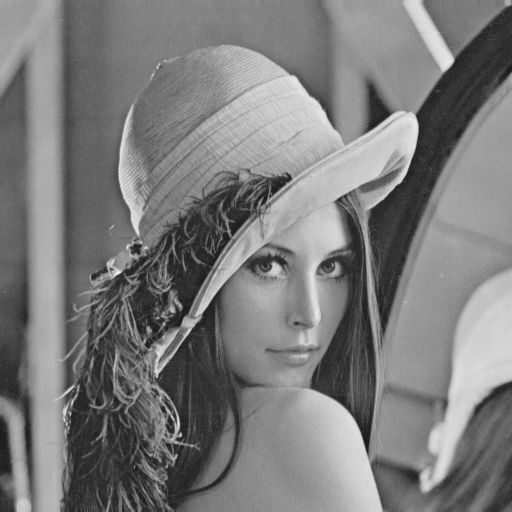}\hfill
		\includegraphics[height=4cm]{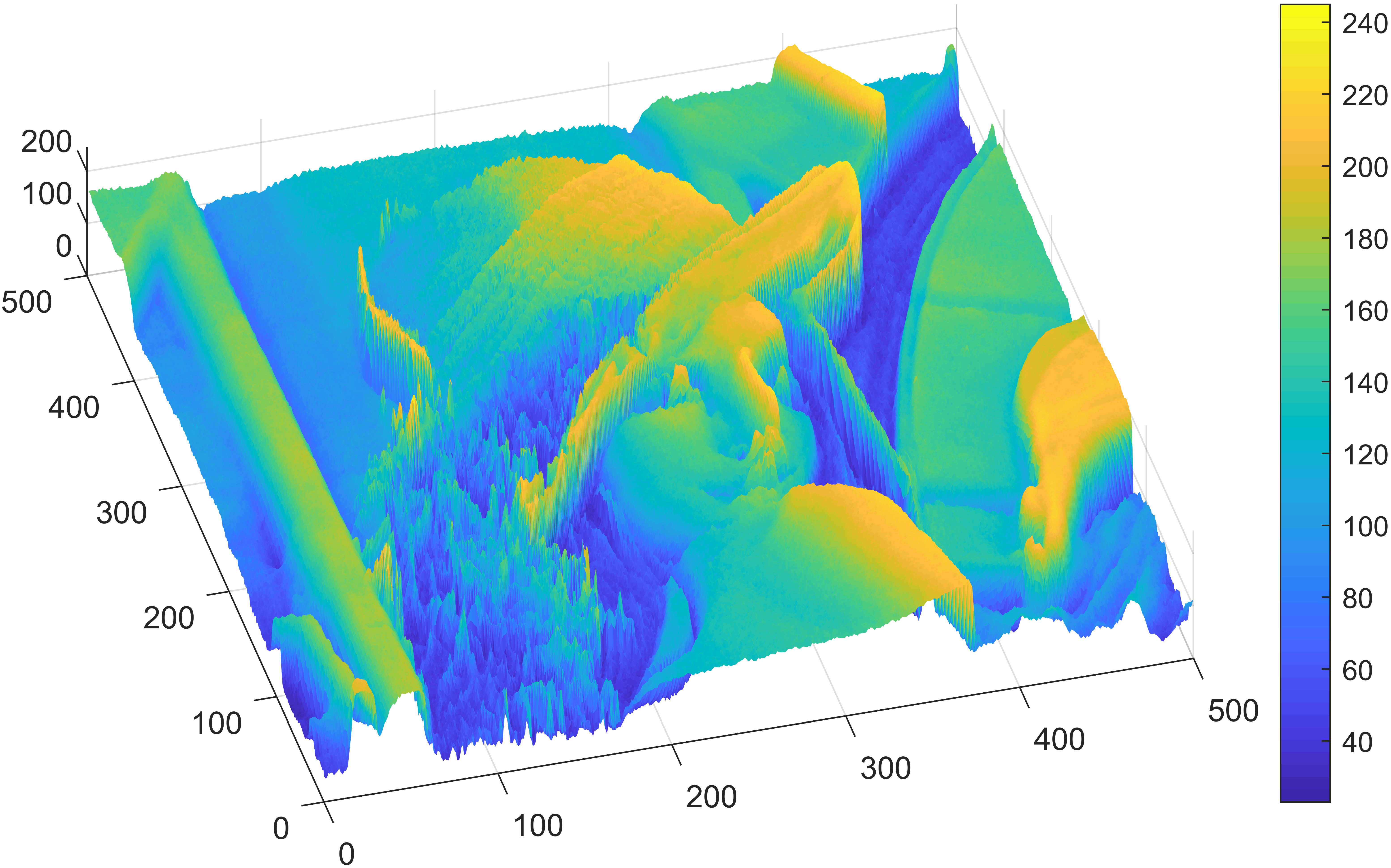}\hfill
	\end{centering}
	\caption{Illustration of the surface representation of a gray scale image with size $400\times 400$ (left). An image can be considered as a two-dimensional surface (right) embedded in three-dimensional space. The pixel value ranges from 0 to 255.\label{fig:1}}
\end{figure}

In terms of Partial Differential Equations (PDEs), the Gaussian convolution is obtained by solving a sequence of the parabolic
PDEs forward in time. Indeed, if we consider
\begin{align}\label{pdeimage}
\begin{cases}
	\partial_{t}u-\Delta u=0 & \text{in }\mathbb{R}^{2}\times\left(0,T\right),\\
	u\left(x,0\right)=g_{0}^{\varepsilon}\left(x\right) & \text{in }\mathbb{R}^{2},
\end{cases}
\end{align}
then its solution can be expressed as $u(x,t) = \left(G_{\sqrt{2t}}*g_0^{\varepsilon}\right)(x)$ for $x\in\mathbb{R}^2,t>0$.

We remark that due to the smoothness property of the parabolic equation, the smoothing process
will destroy image characteristics such as lines and edges; see Figure \ref{fig:2}. Stemming from
the idea in \cite{Gilboa2002,Welk2009}, we combine the forward and backward diffusion equations to enhance images. Henceforth, it brings us back to the initial state reconstruction problem in which the
noisy final condition as our blurred image is $g_{T}^{\varepsilon}=G_{\sqrt{2T}}*g_{0}^{\varepsilon}(x)$. Unique continuity of this problem is studied in \cite{Friedman2008}.

\begin{figure}
	\begin{centering}
		\hfill\includegraphics[height=4cm]{lena_gray}\hfill
		\includegraphics[height=4cm]{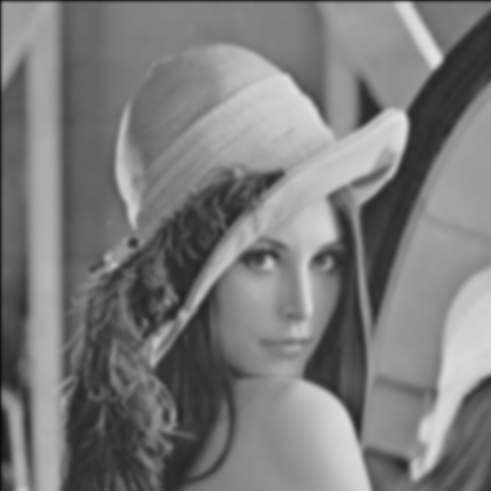}\hfill
	\end{centering}
	\caption{Illustration of the forward process for a color scale image (left) after $t=0.001$. The obtained image (right) is very blurry due to the smoothness property of the parabolic equation. 
		\label{fig:2}}
\end{figure}

\begin{figure}
	\begin{centering}
		\hfill\includegraphics[scale=0.7]{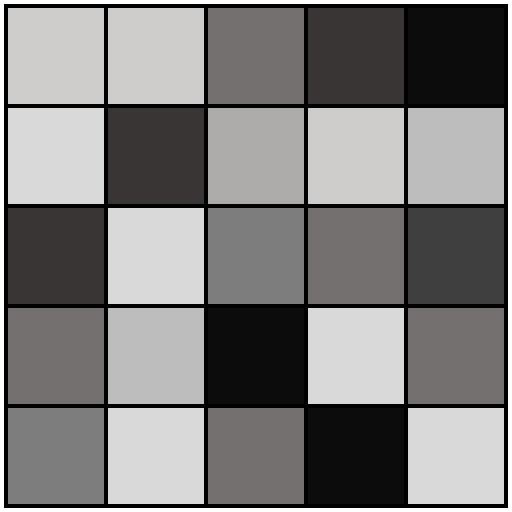}\hfill
		\includegraphics[scale=0.7]{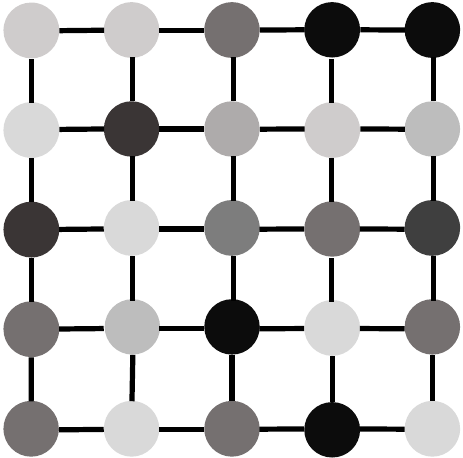}\hfill
	\end{centering}
	\caption{Illustration of how part of a gray scale image (left) is pixelated as a network (right). The value of every vertex is based on the pixel value at the corresponding pixel.
		\label{fig:3}}
\end{figure}

We would like to mention that image can be pixelated as a network; see Figure \ref{fig:3}. Therefore,
one can consider the linear diffusion PDE (\ref{pdeimage}) on graphs. In this work, we concentrate on the simple graph $G=\left(V,E\right)$. It is
an ordered pair of sets, where
\begin{itemize}
	\item $E\subset\left\{ \left\{ x_{i},x_{j}\right\} |x_{i},x_{j}\in V,i\ne j\right\} $
	is a set of undirected edges;
	\item $V$ is a set of vertices $\left\{ x_{i}\right\} _{1\le i\le n}$
	with $n=\left|V\right|<\infty$ being the number of vertices.
\end{itemize}
Consider $u_{i}\left(t\right)=u\left(x_{i},t\right)$ as the pixel value at time $t$ and vertex $x_i$. Note here that the number $n$ should be fixed because of the fixed resolution of the image. The
Graph Laplacian is defined as
$\mathbf{L}=\mathbf{D}-\mathbf{A}$. Here, $\mathbf{A}\in \mathbb{R}^{n\times n}$ is the adjacency matrix, whose the entries $A_{ij}$ are given by
\[
A_{ij}=\begin{cases}
	1 & \text{if }\left\{ x_{i},x_{j}\right\} \in E,\\
	0 & \text{otherwise}.
\end{cases}
\]

From knowledge of the adjacency matrix $\mathbf{A}$, we define the degree $D_{i}=D\left(x_{i}\right)>0$ of a
vertex $x_{i}$ in the simple graph $G$ is the number of vertices
in $G$ that are adjacent to $x_{i}$. As a consequence, it holds
true that
\[
\sum_{j=1}^{n}A_{ij}=D\left(x_{i}\right)=:D_{i}\quad\text{for }1\le i\le n.
\]
By this way, $\mathbf{D}\in \mathbb{R}^{n\times n}$ is the degree matrix which is diagonal with $D_{i}$ being the degree of the $i$-th vertex. Henceforth, $\mathbf{D}\in \mathbb{R}^{n\times n}$ is of the following form:
\[
\mathbf{D}=\begin{bmatrix}D_{1} & 0 & \cdots & 0 & 0\\
	0 & D_{2} & \ddots &  & 0\\
	\vdots & \ddots & \ddots & \ddots & \vdots\\
	0 &  & \ddots & D_{n-1} & 0\\
	0 & 0 & \cdots & 0 & D_{n}
\end{bmatrix}.
\]
In this setting, our Graph Laplacian matrix $\mathbf{L}\in\mathbb{R}^{n\times n}$ is symmetric and its entries are given by
\begin{align*}
	L_{ij}=\begin{cases}
		-1 & \text{if }i\ne j\text{ and }\left\{ x_{i},x_{j}\right\} \in E,\\
		D_{i} & \text{if }i=j,\\
		0 & \text{otherwise}.
	\end{cases}
\end{align*}

Let $U=U\left(t\right)$ now be the $n$-dimensional vector of $u_{i}\left(t\right)$ for $1\le i\le n$, i.e.
\[
U=\begin{pmatrix}u_{1} & u_{2} & \cdots & u_{n-1} & u_{n}\end{pmatrix}^{\text{T}}.
\]
Suppose that $U_T^{\varepsilon}$ is the blurred image $g_T^{\varepsilon}$ pixelated on the simple graph under study. In this scenario, our initial state reconstruction reads as
\begin{align}\label{pdegraph}
	\begin{cases}
		\frac{dU}{dt}\left(t\right)+\mathbf{L}U\left(t\right)=0 & \text{for }t\in\left(0,T\right),\\
		U\left(T\right)=U_{T}^{\varepsilon}\in\mathbb{R}^{n}.
	\end{cases}
\end{align}

%

\section{Fourier instability of the initial state reconstruction}

In this section, we show the natural instability of the initial state reconstruction problem (\ref{pdegraph}). When doing so, we need to introduce the standard inner product
\[
\left\langle A,B\right\rangle =A^{\text{T}}B=\sum_{i=1}^{n}a_{i}b_{i}\quad\text{for }A,B\in\mathbb{R}^{n}.
\]
Hereby, the corresponding $\ell^{2}$-norm $\left\Vert A\right\Vert =\sqrt{\left\langle A,A\right\rangle }$
is given by
\[
\left\Vert A\right\Vert ^{2}=\sum_{i=1}^{n}a_{i}^{2}.
\]

Next, we introduce in Definition \ref{def:1} regarding an eigenvector of a matrix. Then, we state the standard spectral theorem without proof in Theorem \ref{thm:1}. These are essential in our way to show the Fourier instability of problem (\ref{pdegraph}).

\begin{defn}\label{def:1}
	An eigenvector of a matrix $\mathcal{A}\in\mathbb{R}^{n\times n}$
	is a (non-zero) vector $\phi$ such that $\mathcal{A}\phi=\lambda\phi$
	for some scalar $\lambda$. The value $\lambda$ is the corresponding
	eigenvalue, which is a root of the characteristic polynomial of $\mathcal{A}$,
	$p_{\mathcal{A}}=\det\left(\lambda I-\mathcal{A}\right)$ where $I\in\mathbb{R}^{n\times n}$
	denotes the identity matrix.
\end{defn}

\begin{thm}\label{thm:1}
	Let $\mathbf{M}\in\mathbb{R}^{n\times n}$ be a real and symmetric
	matrix. Then it has $n$ orthogonal eigenvectors and the eigenvalues
	are real and non-negative.
\end{thm}

Since the Graph Laplacian matrix $\mathbf{L}$ is real and symmetric, we know
that there is an orthonormal basis $\phi_{1},\ldots,\phi_{n}$ of
$\mathbb{R}^{n}$ such that each $\phi_{j}$ is an eigenvector of
$\mathbf{L}$. Let $\lambda_{j}$ be the real eigenvalue corresponding
to $\phi_{j}$, i.e. $\mathbf{L}\phi_{j}=\lambda_{j}\phi_{j}$ for
$1\le j\le n$. Moreover, we have $0\le \lambda_1 \le \lambda_2 \le \ldots \le \lambda_{n}$, and the sum of these eigenvalues  is twice the number of edges of the simple graph under consideration  (cf. \cite{Brouwer2011}). We find $U\left(t\right)$ as a linear combination
of the eigenvectors:
\begin{equation}
	U\left(t\right)=\sum_{i=1}^{n}B_{i}\left(t\right)\phi_{i}.\label{eq:spectral}
\end{equation}

It is clear that in this form, $U(t)\in \mathbb{R}^{n}$ for every $t$. Now, we multiply both sides of (\ref{pdegraph}) (in the sense of the standard
product) by $\phi_{j}\in\mathbb{R}^{n}$ and thus, arrive at
\[
\frac{d}{dt}\left\langle U\left(t\right),\phi_{j}\right\rangle +\left\langle \mathbf{L}U\left(t\right),\phi_{j}\right\rangle =0,
\]
equivalently,
\[
\frac{d}{dt}\left\langle U\left(t\right),\phi_{j}\right\rangle +\left\langle U\left(t\right),\mathbf{L}\phi_{j}\right\rangle =0.
\]
Since $\mathbf{L}\phi_{j}=\lambda_{j}\phi_{j}$, we obtain the following
differential equation:
\begin{equation}
	\frac{d}{dt}\left\langle U\left(t\right),\phi_{j}\right\rangle +\lambda_{j}\left\langle U\left(t\right),\phi_{j}\right\rangle =0,\label{eq:5}
\end{equation}
associated with the following final condition (cf. (\ref{pdegraph})),
\begin{equation}
	\left\langle U\left(T\right),\phi_{j}\right\rangle =\left\langle U_{T}^{\varepsilon},\phi_{j}\right\rangle .\label{eq:8}
\end{equation}
Solving (\ref{eq:5})-(\ref{eq:8}) as an initial-value differential
problem, we obtain
\begin{equation}
	\left\langle U\left(t\right),\phi_{j}\right\rangle =e^{\lambda_{j}(T-t)}\left\langle U_{T}^{\varepsilon},\phi_{j}\right\rangle .\label{eq:9}
\end{equation}
Then plugging the spectral form of $U\left(t\right)$ (cf. (\ref{eq:spectral}))
in (\ref{eq:9}) and using the orthogonality of $\phi_{i}$ yield
\[
e^{\lambda_{j}(T-t)}\left\langle U_{T}^{\varepsilon},\phi_{j}\right\rangle =\left\langle U\left(t\right),\phi_{j}\right\rangle =\left\langle \sum_{i=1}^{n}B_{i}^{\varepsilon}\left(t\right)\phi_{i},\phi_{j}\right\rangle =\sum_{i=1}^{n}B_{i}^{\varepsilon}\left(t\right)\left\langle \phi_{i},\phi_{j}\right\rangle =B_{j}^{\varepsilon}\left(t\right).
\]

Henceforth, our solution $U\left(t\right)$ to the initial state reconstruction problem (\ref{pdegraph})
can be computed as
\begin{align}\label{sol1}
U\left(t\right)=\sum_{i=1}^{n}e^{\lambda_{i}(T-t)}\left\langle U_{T}^{\varepsilon},\phi_{i}\right\rangle \phi_{i}.
\end{align}
Hence, the initial information of an pixelated noisy image, denoted by $U_0^{\varepsilon} = U^{\varepsilon}(0)$ is formulated by
\begin{align}\label{sol2}
	U^{\varepsilon}_{0}=\sum_{i=1}^{n}e^{\lambda_{i}T}\left\langle U_{T}^{\varepsilon},\phi_{i}\right\rangle \phi_{i}.
\end{align}

Observe in (\ref{sol1}) that the exponential kernel $e^{\lambda_{i}(T-t)}$ is the main natural factor of the Fourier instability. In this sense, the solution is exponentially unstable when the number of vertices, which is $n$, becomes larger and larger for a better resolution of image. If we keep refining the mesh of discretization or, in other words, increasing the number of vertices, the reconstruction process will be worse regardless of any numerical method applied. 


Our aim in this work is to study a stable approximate solution of the discrete model  on graphs (\ref{pdegraph}). This new research
will formulate a new numerical approach in image denoising among well known approaches
such as Perona--Malik \cite{Shi2021}, $p$-Laplacian \cite{Shi2021a}, Mumford--Shah \cite{Jung2011} and deep
learning \cite{Tian2020}.


\section{A spectral graph-based approximate solution}

For each noise level $\varepsilon$, we consider the set of admissible eigenvalues $\Theta\left(\varepsilon\right)=\left\{ i\in\mathbb{N}^{*}:\lambda_{i}\le M_{\varepsilon}\right\} $. Here, $M_{\varepsilon} = M(\varepsilon)>0$ is the so-called regularization parameter which will be chosen later. It should be that $\left|\Theta(\varepsilon)\right|\le \left|V\right|$ as the resolution of image is fixed. If there exists $\varepsilon_0$ such that $\left|\Theta(\varepsilon_0)\right|> \left|V\right|$, for any $\varepsilon\ge \varepsilon_0$ we set $M_{\varepsilon} = \lambda_n$, where we recall that $\lambda_n$ is the largest eigenvalue of the Graph Laplacian matrix. We thereby construct the following cut-off projection in which the number of vertices $n$ is selected appropriately in terms of $\varepsilon$.
\begin{align}\label{sol3}
	\mathbf{P}^{\varepsilon}U\left(t\right)=\sum_{i\in \Theta\left(\varepsilon\right)}e^{\lambda_{i}(T-t)}\left\langle U(T),\phi_{i}\right\rangle \phi_{i}.
\end{align}

\begin{remark}
	The underlying cut-off projection is one of the conventional Fourier-based approaches that is ``computable'' in the regularization theory; cf. \cite{NTT10,Tuan2018}. We remark that the existing literature on regularization of terminal-value problems is huge. For brevity, we mention here some PDE-based approaches (e.g. quasi-reversibility method \cite{LL67,Nguyen2019}, quasi-boundary value method \cite{Denche2005}, and references cited therein) and some minimization-based approaches (see in \cite{Fan2019} regarding methods of total variation, sparse representation and others, and convexification in \cite{Klibanov2019}).
\end{remark}


\begin{thm}\label{thm:2}
	Consider $U(T)\in\mathbb{R}^{n}$. Then, it holds true that
	\[
	\left\Vert \mathbf{P}^{\varepsilon}U\left(t\right)\right\Vert \le e^{M_{\varepsilon}\left(T-t\right)}\left\Vert U(T)\right\Vert,
	\]
\end{thm}
\begin{proof}
	The proof can be done directly by taking into account the formulation of the cut-off projection in (\ref{sol3}). Indeed, by the  orthogonality of the eigenvectors, we have
	\begin{align*}
		\left\Vert \mathbf{P}^{\varepsilon}U\left(t\right)\right\Vert ^{2} & =\left\langle \mathbf{P}^{\varepsilon}U\left(t\right),\mathbf{P}^{\varepsilon}U\left(t\right)\right\rangle \\
		& =\sum_{i\in\Theta(\varepsilon)}\sum_{j\in\Theta(\varepsilon)}e^{\lambda_{i}\left(T-t\right)}e^{\lambda_{j}\left(T-t\right)}\left\langle U(T),\phi_{i}\right\rangle \left\langle U(T),\phi_{j}\right\rangle \left\langle \phi_{i},\phi_{j}\right\rangle \\
		& =\sum_{i\in\Theta(\varepsilon)}e^{2\lambda_{i}\left(T-t\right)}\left|\left\langle U(T),\phi_{i}\right\rangle \right|^{2}.
	\end{align*}
Thus, it yields that
\[
\left\Vert \mathbf{P}^{\varepsilon}U\left(t\right)\right\Vert ^{2}\le e^{2M_{\varepsilon}\left(T-t\right)}\left\Vert U(T)\right\Vert ^{2},
\]
which completes the proof of the theorem.
\end{proof}

Theorem \ref{thm:2} shows that the cut-off projection solution is bounded at each noise
level. This result also proves the stability of the approximate solution when the input, $U(T)$, is perturbed by noise. Let $U_0\in \mathbb{R}^n$ be the pixel values of the ideal image $g_0$, and let $U_0^{\varepsilon}\in \mathbb{R}^n$ be the pixel values of the noisy image $g_0^{\varepsilon}$. We assume that
\[
\left\Vert U_{0}-U_{0}^{\varepsilon}\right\Vert \le\varepsilon.
\]
Henceforth, we obtain
\begin{align}\label{noise}
\left\Vert U_{T}-U_{T}^{\varepsilon}\right\Vert \le\left\Vert U_{0}-U_{0}^{\varepsilon}\right\Vert \le\varepsilon,
\end{align}
where $U_{T}$ and $U_{T}^{\varepsilon}$ are, respectively, the terminal pixel values of the ideal and noisy images after the forward solver (\ref{pdeimage}).

\begin{thm}\label{thm:3}
	Let $U_T,U_T^{\varepsilon}\in\mathbb{R}^{n}$ be  the terminal data of the ideal solution $U(t)$ and the noisy solution $U^{\varepsilon}(t)$, respectively. These data satisfy (\ref{noise}). Then, the following stability estimate holds true:
	\begin{align}\label{ddd}
		\left\Vert \mathbf{P}^{\varepsilon}U\left(t\right)-\mathbf{P}^{\varepsilon}U^{\varepsilon}\left(t\right)\right\Vert \le e^{M_{\varepsilon}\left(T-t\right)}\varepsilon.
	\end{align}
\end{thm}
\begin{proof}
	The proof is straightforward because
	\begin{align*}
		\left\Vert \mathbf{P}^{\varepsilon}U\left(t\right)-\mathbf{P}^{\varepsilon}U^{\varepsilon}\left(t\right)\right\Vert ^{2} & =\sum_{i\in\Theta(\varepsilon)}e^{2\lambda_{i}\left(T-t\right)}\left|\left\langle U_{T}-U_{T}^{\varepsilon},\phi_{i}\right\rangle \right|^{2}\\
		& \le e^{2M_{\varepsilon}\left(T-t\right)}\varepsilon^{2}.
	\end{align*}
\end{proof}

Now, we are in a position to state the following theorem for the convergence rate of the cut-off projection.

\begin{thm}\label{thm:4}
	Let $\mathbf{P}^{\varepsilon}U^{\varepsilon}(t)\in \mathbb{R}^{n}$ be the stable approximate solution of $U(t)$ associated with the terminal data $U_{T}^{\varepsilon}$. Then, the following uniform-in-time estimate holds:
	\begin{align}\label{eee}
		\left\Vert U\left(t\right)-\mathbf{P}^{\varepsilon}U^{\varepsilon}\left(t\right)\right\Vert \le M_{\varepsilon}^{-1}\left\Vert \frac{d}{dt}U\left(t\right)\right\Vert +e^{M_{\varepsilon}\left(T-t\right)}\varepsilon.
	\end{align}
	Furthermore, by choosing
	\begin{align*}
		M_{\varepsilon} = \frac{1}{T}\ln\left(\varepsilon^{-\gamma}\right),\quad \gamma \in (0,1),
	\end{align*}
	we obtain
	\begin{align}\label{fff}
		\left\Vert U\left(t\right)-\mathbf{P}^{\varepsilon}U^{\varepsilon}\left(t\right)\right\Vert \le \frac{T}{\ln\left(\varepsilon^{-\gamma}\right)}\left\Vert \frac{d}{dt}U\left(t\right)\right\Vert +\varepsilon^{1-\gamma\left(1-\frac{t}{T}\right)}.
	\end{align}
\end{thm}
\begin{proof} In view of the facts that $\left|\Theta(\varepsilon)\right|\le \left|V\right|$ and
	\begin{align}\label{aaa}
	U\left(t\right)=\sum_{i=1}^{n}e^{\lambda_{i}(T-t)}\left\langle U_{T},\phi_{i}\right\rangle \phi_{i},
	\end{align}
	we have
	\begin{align}\label{bbb}
	U\left(t\right)-\mathbf{P}^{\varepsilon}U\left(t\right)=\sum_{i\notin\Theta(\varepsilon)}e^{\lambda_{i}\left(T-t\right)}\left\langle U_{T},\phi_{i}\right\rangle \phi_{i}
	\end{align}
	
	It follows from (\ref{aaa}) that
	\[
	\frac{d}{dt}U\left(t\right)=\sum_{i=1}^{n}-\lambda_{i}e^{\lambda_{i}(T-t)}\left\langle U_{T},\phi_{i}\right\rangle \phi_{i}.
	\]
	Combining this with (\ref{bbb}), we estimate that
	\begin{align*}
		\left\Vert U\left(t\right)-\mathbf{P}^{\varepsilon}U\left(t\right)\right\Vert ^{2} & =\sum_{i\notin\Theta\left(\varepsilon\right)}\lambda_{i}^{-2}\lambda_{i}^{2}e^{2\lambda_{i}\left(T-t\right)}\left|\left\langle U_{T},\phi_{i}\right\rangle \right|^{2}.\\
		& \le M_{\varepsilon}^{-2}\left\Vert \frac{d}{dt}U\left(t\right)\right\Vert ^{2}.
	\end{align*}
	It is equivalent to
	\begin{align}\label{ccc}
	\left\Vert U\left(t\right)-\mathbf{P}^{\varepsilon}U\left(t\right)\right\Vert \le M_{\varepsilon}^{-1}\left\Vert \frac{d}{dt}U\left(t\right)\right\Vert .
	\end{align}
	Henceforth, by combining (\ref{ccc}), (\ref{ddd}), and using the triangle inequality, we obtain
	\begin{align*}
		\left\Vert U\left(t\right)-\mathbf{P}^{\varepsilon}U^{\varepsilon}\left(t\right)\right\Vert  & \le\left\Vert U\left(t\right)-\mathbf{P}^{\varepsilon}U\left(t\right)\right\Vert +\left\Vert \mathbf{P}^{\varepsilon}U\left(t\right)-\mathbf{P}^{\varepsilon}U^{\varepsilon}\left(t\right)\right\Vert \\
		& \le M_{\varepsilon}^{-1}\left\Vert \frac{d}{dt}U\left(t\right)\right\Vert +e^{M_{\varepsilon}\left(T-t\right)}\varepsilon.
	\end{align*}
	This completes the proof of (\ref{eee}). Proof of (\ref{fff}) also follows.
\end{proof}

\begin{remark}
	In the proof of Theorem \ref{thm:4}, we obtain a logarithmic rate of convergence under the assumption that $\left|\Theta(\varepsilon)\right|\le \left|V\right|$. When $\varepsilon$ is so small such that $\left|\Theta(\varepsilon)\right|> \left|V\right|$, as mentioned above, we set $M_{\varepsilon} = \lambda_{n}$. This means that $U(t)$ coincides with $\mathbf{P}^{\varepsilon}U(t)$, i.e. $U(t) - \mathbf{P}^{\varepsilon}U(t) = 0$. In this scenario, we obtain a Lipschitz rate of convergence. Last but not least, taking $t=0$ in (\ref{fff}) gives us that
	\begin{align*}
		\left\Vert U_0-\mathbf{P}^{\varepsilon}U_0^{\varepsilon}\right\Vert \le \frac{T}{\ln\left(\varepsilon^{-\gamma}\right)}\left\Vert \frac{d}{dt}U_0\right\Vert +\varepsilon^{1-\gamma}
	\end{align*}
	in which we assume the existence of $\frac{d}{dt}U_0\in \mathbb{R}^{n}$.
\end{remark}

\section{Numerical examples}

In this section, we provide a series of numerical examples to demonstrate the proposed reconstruction algorithm in denoising gray images corrupted by the commonly used additive white Gaussian noise (AWGN). The algorithm is implemented using MATLAB 2016 and executed using a computer equipped with Intel\textsuperscript{\textregistered} Core{\texttrademark} i5-10210U @ 1.60GHz and 8 GB of RAM. In our numerical experiments, we use \emph{Set12} -- a widely-used test dataset for benchmarking denoising algorithms; cf. \cite{Roth2005}. The  dataset consists of 12 gray images of sizes $256 \times 256$ or $512 \times 512$; however, we resize the images to the size of $128 \times 128$ due to the limitation of computer hardware in creating larger graphs. We apply the AWGN with variance  $\sigma_{noise}=20$ to this resized dataset and thereby, obtain noisy images.

The blurred images $U_{T}^{\varepsilon}$ are obtained via the forward process in (\ref{pdeimage}). For simplicity, we apply in this process the standard forward Euler method with the Courant number of 0.03 for the conditional stability and with $T = 0.5$ that corresponds to the deviation $\sigma = 1$ in the Gaussian blurring process.  Since we work with $128 \times 128$ images, the number of vertices is $n=16384$ in our numerical results. Moreover, we fix $\varepsilon = 0.1$ and $\gamma = 0.5$ that gives $M_{\varepsilon} \approx 2.30$. For objective assessment, we use peak signal-to-noise ratio (PSNR) as an image quality metric in evaluating denoising performance, which is represented in the logarithmic decibel (dB) scale. In principle, it is computed as
\begin{equation}
	\mathrm{PSNR}(U_0, \mathbf{P}U^{\varepsilon}_{0}) = 10\log_{10}\left(\frac{255^2}{\frac{1}{n}\sum_{i=1}^n (u_{0i}- \mathbf{P}u^{\varepsilon}_{0i})^2}\right).
\end{equation}
To estimate the PSNR of our images, we compare those to the original ones that are resized from the \textit{Set12} dataset. This comparison is operated by the built-in MATLAB function \texttt{psnr}.

Figures ~\ref{fig:results-1} and~\ref{fig:results-2} show the denoising results of the cut-off solution $\mathbf{P}U_0^{\varepsilon}$ in (\ref{sol3}) for all images in \emph{Set12}. The denoised results exhibit better visual quality than that of noisy images. Especially in regions with rich textures, the noise is reduced while edge features are preserved without significant degradation. Table~\ref{tab:psnr} shows the subjective evaluations for denoising results of all images in \emph{Set12} from approximation solution (\ref{sol2}) and cut-off solution $\mathbf{P}U_0^{\varepsilon}$ in (\ref{sol3}) in terms of PSNR. As can be seen from the table, the cut-off solution $\mathbf{P}U_0^{\varepsilon}$ outperforms the approximation solution (\ref{sol2}), as the PSNR scores of cut-off solutions $\mathbf{P}U_0^{\varepsilon}$ are higher than those of approximation solutions from 1.66dB (for \textit{Monarch} image) to 2.64dB (for \textit{Couple} image). We detail that 1.66dB for the \textit{Monarch} image is calculated directly from the PSNR difference between the cut-off (18.86) and approximation (17.20) solutions; see Table \ref{tab:psnr}. The average of PSNR increase is 2.17dB, i.e. 20.31dB of cut-off solutions $\mathbf{P}U_0^{\varepsilon}$ compared with 18.14dB of approximation solutions (\ref{sol2}), which is a significant improvement.

Last but not least, we provide Table \ref{tab:time} to show the time efficiency of the cut-off projection under study. As can readily be expected from the fact that $\left|\Theta(\varepsilon)\right|<\left|V\right|$, in average, the running time for the cut-off solution $\mathbf{P}U_0^{\varepsilon}$ is 4.5 times less than that of the approximation (\ref{sol2}). Therefore, it is potential to adapt the cut-off projection to more complex scenarios (e.g. color images with three-dimensional matrices) in upcoming research projects.

\begin{table}[!t]
	\begin{center}
		\small
		\setlength{\tabcolsep}{4.75pt}
		\caption{Quantitative comparison in terms of PSNR for denoising performance between approximate solutions (\ref{sol2}) and cut-off solutions (\ref{sol3}).}
		\label{tab:psnr}
		\begin{tabular}{l||cc}
			\hline
			Image name & Approximation (\ref{sol2}) & Cut-off (\ref{sol3})\\
			\hline
			\hline
			Camera man & 17.89 & 20.15 \\
			\hline
			House & 19.40 & 22.00 \\
			\hline
			Peppers & 18.10 & 20.06 \\
			\hline
			Starfish & 18.05 & 20.08 \\
			\hline
			Monarch & 17.20 & 18.86 \\
			\hline
			Airplane & 17.76 & 19.83 \\
			\hline
			Parrot & 16.66 & 18.39 \\
			\hline
			Lena & 18.42 & 20.66 \\
			\hline
			Barbara & 18.34 & 20.37 \\
			\hline
			Boat & 18.50 & 20.94 \\
			\hline
			Man & 18.75 & 21.09 \\
			\hline
			Couple & 18.64 & 21.28 \\
			\hline
			\hline
			Avg. & 18.14 & 20.31
		\end{tabular}
	\end{center}
\end{table}

\begin{table}[!t]
	\begin{center}
		\small
		\setlength{\tabcolsep}{4.75pt}
		\caption{Comparison of running time in seconds between approximate solutions (\ref{sol2}) and cut-off solutions (\ref{sol3}).}
		\label{tab:time}
		\begin{tabular}{l||cc}
			\hline
			Image name & Approximation (\ref{sol2}) & Cut-off (\ref{sol3})\\
			\hline
			\hline
			Camera man & 1.23 & 0.28 \\
			\hline
			House & 1.15 & 0.27 \\
			\hline
			Peppers & 1.15 & 0.24 \\
			\hline
			Starfish & 1.13 & 0.26 \\
			\hline
			Monarch & 1.14 & 0.25 \\
			\hline
			Airplane & 1.14 & 0.25 \\
			\hline
			Parrot & 1.13 & 0.26 \\
			\hline
			Lena & 1.13 & 0.26 \\
			\hline
			Barbara & 1.16 & 0.27 \\
			\hline
			Boat & 1.16 & 0.26 \\
			\hline
			Man & 1.26 & 0.28 \\
			\hline
			Couple & 1.20 & 0.27 \\
			\hline
			\hline
			Avg. & 1.17 & 0.26
		\end{tabular}
	\end{center}
\end{table}

\begin{figure}
	\centering
	\subfloat[\textit{Camera man} image] {\includegraphics[width=0.23\linewidth]{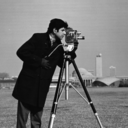}\
	\includegraphics[width=0.23\linewidth]{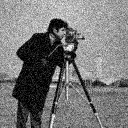}\
	\includegraphics[width=0.23\linewidth]{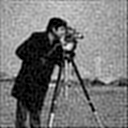}}\\
	\vspace{-3mm}
	\subfloat[\textit{House} image] {\includegraphics[width=0.23\linewidth]{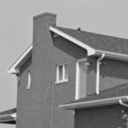}\
	\includegraphics[width=0.23\linewidth]{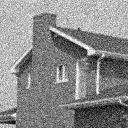}\
	\includegraphics[width=0.23\linewidth]{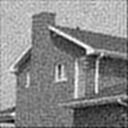}}\\
	\vspace{-3mm}
	\subfloat[\textit{Peppers} image] {\includegraphics[width=0.23\linewidth]{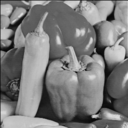}\
	\includegraphics[width=0.23\linewidth]{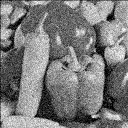}\
	\includegraphics[width=0.23\linewidth]{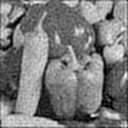}}\\
	\vspace{-3mm}
	\subfloat[\textit{Starfish} image] {\includegraphics[width=0.23\linewidth]{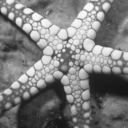}\
	\includegraphics[width=0.23\linewidth]{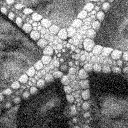}\
	\includegraphics[width=0.23\linewidth]{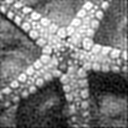}}\\
	\vspace{-3mm}
	\subfloat[\textit{Monarch} image] {\includegraphics[width=0.23\linewidth]{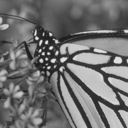}\
	\includegraphics[width=0.23\linewidth]{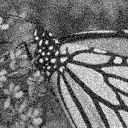}\
	\includegraphics[width=0.23\linewidth]{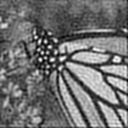}}\\
	\vspace{-3mm}
	\subfloat[\textit{Airplane} image] {\includegraphics[width=0.23\linewidth]{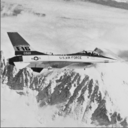}\
	\includegraphics[width=0.23\linewidth]{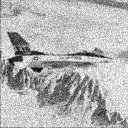}\
	\includegraphics[width=0.23\linewidth]{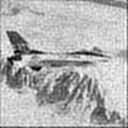}}\\
	\vspace{-3mm}
	
	\caption
	{
		Visual comparison of denoising results on standard test images in the \textit{Set12} dataset. Left column: original images; middle column: noisy images; right column: denoised images.
	}
	\label{fig:results-1}
\end{figure}

\begin{figure}
	\centering
	\subfloat[\textit{Parrot} image] {\includegraphics[width=0.23\linewidth]{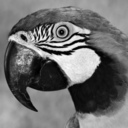}\
	\includegraphics[width=0.23\linewidth]{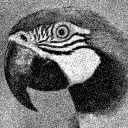}\
	\includegraphics[width=0.23\linewidth]{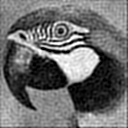}}\\
	\vspace{-3mm}
	\subfloat[\textit{Lena} image] {\includegraphics[width=0.23\linewidth]{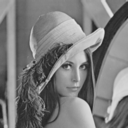}\
		\includegraphics[width=0.23\linewidth]{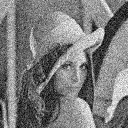}\
		\includegraphics[width=0.23\linewidth]{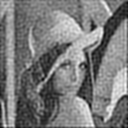}}\\
	\vspace{-3mm}
	\subfloat[\textit{Barbara} image] {\includegraphics[width=0.23\linewidth]{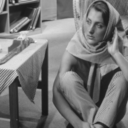}\
		\includegraphics[width=0.23\linewidth]{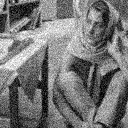}\
		\includegraphics[width=0.23\linewidth]{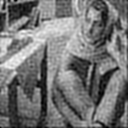}}\\
	\vspace{-3mm}
	\subfloat[\textit{Boat} image] {\includegraphics[width=0.23\linewidth]{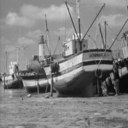}\
		\includegraphics[width=0.23\linewidth]{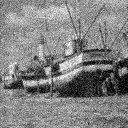}\
		\includegraphics[width=0.23\linewidth]{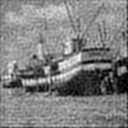}}\\
	\vspace{-3mm}
	\subfloat[\textit{Man} image] {\includegraphics[width=0.23\linewidth]{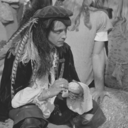}\
		\includegraphics[width=0.23\linewidth]{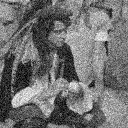}\
		\includegraphics[width=0.23\linewidth]{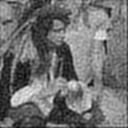}}\\
	\vspace{-3mm}
	\subfloat[\textit{Couple} image] {\includegraphics[width=0.23\linewidth]{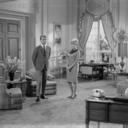}\
		\includegraphics[width=0.23\linewidth]{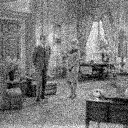}\
		\includegraphics[width=0.23\linewidth]{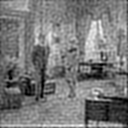}}\\
	\vspace{-3mm}
	
	\caption
	{
		Visual comparison of denoising results on standard test images in the \textit{Set12} dataset. Left column: original images; middle column: noisy images; right column: denoised images.
	}
	\label{fig:results-2}
\end{figure}

\subsection*{Acknowledgment}

Vo Anh Khoa would like to thank Professor Paul Sacks (Iowa State University, USA) for great support of his research career.

\bibliographystyle{plain}{\bibliography{bibtex}}

\begin{thebibliography}{10}

\bibitem{Brouwer2011}
A.~E. Brouwer and W.~H. Haemers.
\newblock {\em Spectra of {G}raphs}.
\newblock Springer New York, December 2011.

\bibitem{Denche2005}
M.~Denche and K.~Bessila.
\newblock A modified quasi-boundary value method for ill-posed problems.
\newblock {\em Journal of Mathematical Analysis and Applications},
  301(2):419--426, 2005.

\bibitem{Fan2019}
L.~Fan, F.~Zhang, H.~Fan, and C.~Zhang.
\newblock Brief review of image denoising techniques.
\newblock {\em Visual Computing for Industry, Biomedicine, and Art}, 2(1),
  2019.

\bibitem{Friedman2008}
A.~Friedman.
\newblock {\em Partial {D}ifferential {E}quations of {P}arabolic {T}ype}.
\newblock Dover Pubn Inc, April 2008.

\bibitem{Gilboa2002}
G.~Gilboa, N.~Sochen, and Y.~Y. Zeevi.
\newblock Forward-and-backward diffusion processes for adaptive image
  enhancement and denoising.
\newblock {\em IEEE Transactions on Image Processing}, 11(7):689--703, 2002.

\bibitem{Jung2011}
M.~Jung, X.~Bresson, T.~F. Chan, and L.~A. Vese.
\newblock Nonlocal {M}umford-{S}hah regularizers for color image restoration.
\newblock {\em IEEE Transactions on Image Processing}, 20(6):1583--1598, 2011.

\bibitem{Klibanov2019}
M.~V. Klibanov and A.~G. Yagola.
\newblock Convergent numerical methods for parabolic equations with reversed
  time via a new {C}arleman estimate.
\newblock {\em Inverse Problems}, 35(11):115012, 2019.

\bibitem{LL67}
R.~Latt\`es and J.~L. Lions.
\newblock {\em M\'ethode de {Q}uasi-r\'eversibilit\'e et {A}pplications}.
\newblock Dunod, Paris, 1967.

\bibitem{NTT10}
P.~T. Nam, D.~D. Trong, and N.~H. Tuan.
\newblock The truncation method for a two-dimensional nonhomogeneous backward
  heat problem.
\newblock {\em Applied Mathematics and Computation}, 216:3423--3432, 2010.

\bibitem{Nguyen2019}
H.~T. Nguyen, V.~A. Khoa, and V.~A. Vo.
\newblock Analysis of a quasi-reversibility method for a terminal value
  quasi-linear parabolic problem with measurements.
\newblock {\em {SIAM} Journal on Mathematical Analysis}, 51(1):60--85, 2019.

\bibitem{Roth2005}
S.~Roth and M.J. Black.
\newblock Fields of experts: a framework for learning image priors.
\newblock In {\em Proceedings of the IEEE/CVF Conference on Computer Vision and
  Pattern Recognition (CVPR)}, volume~2, pages 860--867, 2005.

\bibitem{Shi2021a}
K.~Shi.
\newblock Coupling local and nonlocal diffusion equations for image denoising.
\newblock {\em Nonlinear Analysis: Real World Applications}, 62:103362, 2021.

\bibitem{Shi2021}
K.~Shi.
\newblock Image denoising by nonlinear nonlocal diffusion equations.
\newblock {\em Journal of Computational and Applied Mathematics}, 395:113605,
  2021.

\bibitem{Tian2020}
C.~Tian, L.~Fei, W.~Zheng, Y.~Xu, W.~Zuo, and C.-W. Lin.
\newblock Deep learning on image denoising: An overview.
\newblock {\em Neural Networks}, 131:251--275, 2020.

\bibitem{Tuan2018}
N.~H. Tuan, V.~A. Khoa, M.~T.~N. Truong, T.~T. Hung, and M.~N. Minh.
\newblock Application of the cut-off projection to solve a backward heat
  conduction problem in a two-slab composite system.
\newblock {\em Inverse Problems in Science and Engineering}, 27(4):460--483,
  2018.

\bibitem{Welk2009}
M.~Welk, G.~Gilboa, and J.~Weickert.
\newblock Theoretical {F}oundations for {D}iscrete {F}orward-and-{B}ackward
  {D}iffusion {F}iltering.
\newblock pages 527--538. Springer Berlin Heidelberg, 2009.

\end{thebibliography}

\end{document}